\documentclass[a4paper]{article}
\usepackage[utf8]{inputenc}
\usepackage[pdftex]{hyperref}
\usepackage[pdftex]{graphicx}
\usepackage{amsmath, amsfonts, amsthm, amssymb}
\usepackage{physics}
\usepackage{verbatim}
\usepackage{color}
\usepackage{soul}

\title{On Zhu's algebra and $C_2$--algebra for symplectic fermion vertex algebra $SF(d)^+$ }
\date{\today}
\author{Dražen Adamović and Ante Čeperić}

\newcommand{\Addresses}{{
  \bigskip
  \footnotesize

  D. Adamovi\' c, \textsc{ Department of Mathematics,  Faculty of Science, University of Zagreb,  Croatia}\par\nopagebreak
  \textit{E-mail address:} \texttt{adamovic@math.hr}

  \medskip

  A. \v Ceperi\' c, \textsc{Department of Mathematics,  Faculty of Science, University of Zagreb,  Croatia}\par\nopagebreak
  \textit{E-mail address:} \texttt{ante.ceperic@math.hr}

}}

\theoremstyle{plain}
\newtheorem{thm}{Theorem}[section]
\newtheorem{lem}[thm]{Lemma}
\newtheorem{prop}[thm]{Proposition}
\newtheorem{cor}[thm]{Corollary}
\newtheorem{rem}[thm]{Remark}

\numberwithin{equation}{section}

\DeclareMathOperator{\sspan}{span}

\DeclareMathOperator{\End}{End}
\DeclareMathOperator{\im}{im}

\DeclareMathOperator{\gr}{gr}
\DeclareMathOperator{\Sym}{Sym}

\begin{document}
 
\maketitle
 
\begin{abstract}
  In this paper, we study the family of vertex operator algebras $SF(d)^+$,
  known as symplectic fermions. This family is of a particular interest
  because these VOAs are irrational and $C_2$-cofinite.
  We determine the Zhu's algebra $A(SF(d)^+)$ and show that
  the equality of dimensions of $A(SF(d)^+)$ and the $C_2$--algebra
  $\mathcal P(SF(d)^+)$ holds for $d \geq 2$ (the case of $d=1$
  was treated by T. Abe in \cite{Abe}).
  We use these results to prove a conjecture by Y. Arike and K. Nagatomo
  (\cite{AN}) on the dimension of the space of one-point functions
  on $SF(d)^+$.
\end{abstract}
\tableofcontents
 
\section{Introduction}
Symplectic fermions appeared first in physics literature in the papers by H. G. Kausch \cite{Ka00}, H. G. Kausch and M. Gaberdiel \cite{GK96}, \cite{GK99} and in the mathematical literature in the paper by T. Abe \cite{Abe}.
T. Abe proved that symplectic fermion VOAs $SF(d)^+ \ (d \geq 1)$ form a family of irrational $C_2$-cofinite VOAs   at central charge $c=-2d$. Orbifolds of symplectic fermions are studied in detail by T. Creutzig and A. Linshaw \cite{CL}.
 Some CFT aspects of the theory of symplectic fermions, including connection with quantum groups,  were investigated  in \cite{DR}, \cite{GR17}, \cite{R14}.

They are connected to another family of irrational $C_2$-cofinite VOAs,
the family of triplet algebras $\mathcal W(p) \ (p \geq 2)$ (introduced
in \cite{AM} by D. Adamović and A. Milas) by
$$SF(1)^+ \simeq \mathcal W(2).$$

T. Abe in \cite{Abe} proved the following properties of VOA  $SF(d)^+$:
\begin{itemize}
\item  $SF(d)^+$ is  a $C_2$--cofinite irrational VOA;
\item  $SF(d)^+$ has four irreducible modules, and it contains logarithmic modules;
\item $ \displaystyle A(SF(1)^+) \simeq M_2(\mathbb C) \oplus M_2(\mathbb C) \oplus \mathbb C[x]/(x^2) \oplus \mathbb C; $
\item $\dim A(SF(1)^+) = \dim \mathcal P(SF(1)^+)= 11$.
\end{itemize}
Here, $A(V)$ denotes the Zhu's algebra of a VOA $V$, and $\mathcal P(V)$
denotes the $C_2$--algebra (also known as the Poisson algebra) of $V$, both introduced in \cite{Zhu}
by Y. Zhu.

Although the paper \cite{Abe} contains many general structural results on Zhu's algebra $A(SF(d)^+)$, it does not contain a complete description of  $A(SF(d)^+)$ for general rank $d$.
The main goal of our paper is to completely determine the Zhu's algebra $A(SF(d)^+)$ for general rank $d$.

Our method of proof will use the well known fact that for a general VOA $V$ we have
\begin{equation} \label{leq}
  \dim_{\mathbb C} A(V) \leq \dim_{\mathbb C} \mathcal P(V).
\end{equation}

We will show that for $SF(d)^+$, the equality of dimensions in \eqref{leq} holds.
The general problem of determining for which VOAs the equality in \eqref{leq} holds was posed and explored by M. Gaberdiel and T. Gannon in \cite{GG}.

This problem was solved for the family of triplet algebras $\mathcal W(p)$ in \cite{AM11}, and considered for some subalgebras of $\mathcal W(p)$ in \cite{ALM} and \cite{ALM1}.
It is also important to mention papers \cite{FL} and \cite{FFL}, which
consider this problem for some rational VOAs.

To expand on our method, we need some notation:
\begin{align*}
  \mathcal A_d &:= M_{2d}(\mathbb C) \oplus M_{2d}(\mathbb C) \oplus \Lambda^{\text{even}}(V_{2d}) \oplus \mathbb C \\
  n_d &:= \dim_{\mathbb C} \mathcal A_d = 2^{2d-1} + 8d^2 + 1,
\end{align*}
where $V_{2d}$ stands for the standard ($2d$-dimensional) representation of the Lie algebra $\mathfrak{sp}(2d)$.
The main result of this article is:
\begin{thm} \label{uvod1}
  We have $A(SF(d)^+) \simeq \mathcal A_d$ and
  $$\dim_{\mathbb C} A(SF(d)^+) = \dim_{\mathbb C} \mathcal P(SF(d)^+) = n_d.$$
\end{thm}
The strategy of the proof is:
\begin{enumerate}
\item Using the results from \cite{Abe} on representation theory of $SF(d)^+$, obtain the epimorphism of algebras $\pi: A(SF(d)^+) \to \mathcal A_d$.
\item Find the spanning set of $\mathcal P(SF(d)^+)$
  of cardinality $n_d$.
\end{enumerate}
These two steps (together with the  inequality \eqref{leq}) give us:
$$n_d \leq \dim_{\mathbb C} A(SF(d)^+) \leq \mathcal \dim_{\mathbb C} P(SF(d)^+) \leq n_d,$$
from which we conclude the equality of dimensions together with
the fact that the epimorphism $\pi$ is actually an isomorphism, giving
us the full description of the Zhu's algebra $A(SF(d)^+)$ for general rank
$d$.

We also explore some consequences of having the full description of $A(SF(d)^+)$:
\begin{itemize}

   \item We prove  the Arike-Nagatomo  conjecture (\cite{AN}) about the dimension of a vector space of the one-point functions on $SF(d)^+$.

\item  We show  that the  center of   $A(SF(d)^+)$ is isomorphic to  $\Lambda^{\text{even}}(V_{2d}) \oplus \mathbb C \oplus C \oplus C$ (cf. Corollary \ref{fcor}).  

\item We  prove that the invariant subalgebra  $A(SF(d)^+)^{sp(2d)}$  is    generated only  by $[\omega]$ (cf. Proposition \ref{propinv}). This is quite surprising, since  by \cite{KL}, the invariant vertex algebra  $SF(d)^{sp(2d)}$ is isomorphic to  $\mathcal W_{-d-1/2} (sp(2n), f_{prin})$, and its structure is therefore more complicated.
\end{itemize}

\subsection*{Future work}
 Symplectic fermions belong to a class of vertex operator algebras related to the higher rank logarithmic CFT. Other higher rank analogues of the triplet algebras are vertex algebras $\mathcal W(p)_Q$ introduced in \cite{FT} and studied in \cite{AM14},  \cite{Su}. We believe that for $\mathcal W(p)_Q$ we have the equality of dimensions of Zhu's algebra and $C_2$--algebra.
 
We should mention that it is not easy to find examples of logarithmic
VOAs when $\dim_{\mathbb C} A(V) < \dim_{\mathbb C} \mathcal P(V)$.
It seems that such example is $c = 0$ triplet vertex algebra (cf. \cite{AM10} and \cite[Remark 4]{AM11}), but this is still a conjecture.
We hope to study this problem in our forthcoming publications.
\newpage

\section{Main definitions}

We recall the definitions of Zhu's algebra and $C_2$--algebra of a vertex operator algebra $(V, Y, \mathbf 1, \omega)$ (following \cite{Zhu}).
We will write the mode expansion of a vertex operator associated to $a \in V$ as
$$Y(a,z) = \sum_{n \in \mathbb Z} a_n z^{-n-1}. $$
As usual, for the Virasoro element $\omega$ we will write
$$Y(\omega, z) =  \sum_{n \in \mathbb Z} \omega_n z^{-n-1} = \sum_{n \in \mathbb Z} L(n) z^{-n-2}.$$
Per definition of a vertex operator algebra, we have a $L(0)$-grading
$$V = \bigoplus_{n \in \mathbb Z_{\geq 0}} V_n. $$
We will refer to $a \in V_n$ as homogeneous elements and write $\deg a = n$.

To give a definition of the Zhu's algebra, we define two bilinear maps $*: V \times V \to V, \circ : V
\times V \to V$ in the following way:
for homogeneous $a,b \in V$, we put
\begin{align*}
  a * b &= \Res_z \frac{(1+z)^{\deg a}}{z} Y(a,z)b = \sum_{k \geq 0} {\deg a \choose k} a_{-k-1}b\\
  a \circ b &= \Res_z \frac{(1+z)^{\deg a}}{z^2} Y(a,z)b = \sum_{k \geq 0} {\deg a \choose k} a_{-k-2}b.
\end{align*}
We can extend these operations bilinearly to $V$, and denote by $O(V) \subseteq V$ the linear span of elements of the form $a \circ b$.
Denote $A(V):= V/O(V)$. By \cite{Zhu}, this space has a unital associative
algebra structure (with multiplication induced by $*$ and the unit $1 + O(V)$). This algebra is called the Zhu's algebra of $V$ and we will write
$$[a] = a + O(V). $$
It is well known that we can equip $A(V)$ with an increasing filtration
$$F^nA(V) = \left( \bigoplus_{k = 0}^n V_k  \right) \oplus O(V) $$
that gives $A(V)$ the filtered algebra structure.

For $C_2$--algebra, we let
$$C_2(V) = \sspan_{\mathbb C} \{ a_{-2}b : a, b \in V \}, \ \mathcal P(V) = V/C_2(V). $$
We will denote $\overline{a} = a + C_2(V)$.
By \cite{Zhu}, this quotient space $\mathcal P(V)$ has a structure of
a graded commutative Poisson algebra with
\begin{align*}
  \overline{a} \cdot \overline{b} &= \overline{a_{-1}b} \\
  \{\overline{a}, \overline{b} \} &= \overline{a_{0}b},
\end{align*}
and a grading induced by the $L(0)$-grading. It is well known (see Chapter 2 of \cite{Abe}) that we have a natural epimorphism of unital algebras
\begin{equation} \label{eq:epi}
  \gr A(V) \to \mathcal P(V),
\end{equation}
which gives us the inequality of dimensions
\begin{equation}
  \label{eq:ineq}
  \dim_{\mathbb C} A(V) \geq \dim_{\mathbb C} \mathcal P(V).
\end{equation}

It is also important to mention the connection of $V$-modules to
the $A(V)$-modules.
Let $M$ be a weak $V$-module and let
$$\Omega(M) := \{m \in M : a_n m = 0 \text{ for } a \in V_k, n > k-1.\} $$

For the admissible and simple $M$, one can show that $\Omega(M)$ is
exactly the top component of $M$ (see Proposition 3.4. in \cite{DLM}).

\begin{prop}\cite{Zhu}
  Let $V$ be a VOA.
  \begin{enumerate}
  \item Let $M$ be a weak $V$-module. The linear map
    $$o : V \to \End(\Omega(M)), \ o(a) = a_{\deg a - 1}$$
    induces a representation of $A(V)$ on $\Omega(M)$.
  \item If $M$ is an irreducible $V$-module, then $\Omega(M)$ is
    irreducible $A(V)$-module.
  \item The map $M \to \Omega(M)$ is a bijection from the set of
    inequivalent simple admissible $V$-modules to the set of
    inequivalent simple $A(V)$-modules.
  \end{enumerate}
\end{prop}

\newpage
\section{VOA of symplectic fermions}
\subsection{Definitions and basic results}
The VOA of symplectic fermions was defined in \cite{Abe}, and we follow the notation given there. 
First, one considers the symplectic vector space $\mathfrak h$ of dimension
$2d$ with the canonical basis $e^1, e^2, \ldots, e^d, f^1, f^2, \ldots, f^d$
such that
$$\langle e^i, e^j \rangle = \langle f^i, f^j \rangle = 0, \text{ and } \langle e^i, f^j \rangle = - \delta_{i,j}.$$
The subgroup of $GL(\mathfrak h)$ preserving the symplectic
form is the symplectic group $Sp(2d)$.
One can construct a Lie superalgebra and a vertex operator superalgebra from $\mathfrak h$ in a manner that is similar to the construction of Heisenberg VOA.
We denote this VOSA $SF(d)$.
The easiest way to describe it is by saying that it is
strongly and freely generated by the odd generators $e^1, e^2, \ldots, e^d, f^1, f^2, \ldots, f^d$, with the $\lambda$-bracket
$$[e^i_\lambda e^j] = [f^i_\lambda f^j] = 0, \text{ and } [e^i_\lambda f^j] = - \lambda \delta_{i,j} $$
and the Virasoro vector given by $\omega = \sum_{i = 1}^d e^{i}_{-1}f^i. $
One can easily see that we have
$$L(0)e^i = e^i, \ L(0)f^i = f^i, $$
and that we can identify $SF(d)_1$ with $\mathfrak h$.

Now, the VOA of symplectic fermions $SF(d)^+$ is the even part
of the VOSA $SF(d)$. In \cite{Abe} it is shown that both of these algebras are simple and $C_2$-cofinite, and that the following vectors (for $1 \leq i, j \leq d$):
\begin{align*}
  e^{ij} &= e^i_{-1}e^j, \\
  f^{ij} &= f^i_{-1}f^j, \\
  h^{ij} &= e^i_{-1}f^j, \\
  E^{ij} &= \frac{1}{2} \left( e^i_{-2}e^j + e^j_{-2}e^i \right), \\
  F^{ij} &= \frac{1}{2} \left( f^i_{-2}f^j + f^j_{-2}f^i \right), \\
  H^{ij} &= \frac{1}{2} \left( e^i_{-2}f^j + f^j_{-2}e^i \right)
\end{align*}
strongly generate $SF(d)^+$.
Notice that we have
\begin{align*}
  e^{ij} &= - e^{ji}, & f^{ij} &= - f^{ji}, \\
  E^{ij} &= E^{ji}, & F^{ij} &= - F^{ji}.
\end{align*}

In context of our work, it is also important to mention the following
result.

\begin{thm}[\cite{Abe}, \cite{AM11}]
  We have
  $$A(SF(1)^+) \simeq M_2(\mathbb C) \oplus M_2(\mathbb C) \oplus \mathbb
  C[x]/(x^2) \oplus \mathbb C.$$
  Moreover,
  $$\dim A(SF(1)^+) = \dim \mathcal P(SF(1)^+) = 11. $$
\end{thm}

To compare with Theorem \ref{uvod1}, notice that for two-dimensional
$V_2$ we have
$$\Lambda^{\text{even}}(V_2) \simeq \mathbb C[x]/(x^2). $$
\subsection{Automorphisms and derivations of $SF(d)$} \label{subsec:auto}

In this subsection we recall Abe's results on the automorphisms of $SF(d)$.
Denote by $\theta$ the parity operator on $SF(d)$ (the operator acting
as $\pm 1$ on $SF(d)^\pm$).

\begin{thm}[\cite{Abe}] \label{abeautomorfizmi}
  The automorphism groups of $SF(d)$ and $SF(d)^+$ are $Sp(2d, \mathbb C)$
  and $Sp(2d, \mathbb C)/ \langle \theta \rangle$, respectively.
\end{thm}
The important automorphisms to mention are the permutation automorphisms
$$e^i \mapsto e^{\sigma(i)}, f^i \mapsto f^{\sigma(i)}, \text{ for } \sigma \in S_d$$
($S_d$ is here denoting the symmetric group on $d$ letters) and automorphisms $\tau_i, i = 1, \ldots, d$ defined by
$$\tau_i(e^i) = - f^i, \tau_i(f^i) = e^i, \tau_i(e^j) = e^j, \tau_i(f^j) = f^j, $$
for $j = 1, \ldots, d, j \neq i$.

Mostly, we will prefer to use the Lie algebra $\mathfrak{sp}(2d)$, which acts
on $SF(d)$ with derivations. We will need to use some specific
elements of this Lie algebra, so we need to write down one particular
basis of $\mathfrak{sp}(2d)$.
We will use the notation from \cite{FH}.
Let $1 \leq i, j \leq d, i \neq j$:
\begin{align*}
  H_i e^k &= \delta_{ik} e^i, & H_i f^k &= - \delta_{ik} f^i \\
  X_{ij}e^k &= \delta_{jk} e^i, &  X_{ij}f^k &= - \delta_{ik} f^j,\\
  Y_{ij}e^k &= 0,  & Y_{ij}f^k &= \delta_{jk} e^i + \delta_{ik}e^j,\\
  Z_{ij}e^k &= \delta_{jk} f^i + \delta_{ik}f^j, & Z_{ij}f^k &= 0, \\
  U_{i}e^k &= 0, & U_{i}f^k &= \delta_{ik} e^i,\\
  V_{i}e^k &= \delta_{ik} f^i, & V_{i}f^k &= 0.
\end{align*}

\subsection{Representations of $SF(d)^+$} \label{subsec:rep}

In this subsection, we will recall Abe's results on the representations
of $SF(d)^+$.
In \cite{Abe}, the $\theta$-twisted $SF(d)$-module $SF(d)_\theta$ is defined.

\begin{thm}[\cite{Abe}] \label{abeglavni}
  The list $\{SF(d)^\pm, SF(d)_\theta^\pm\}$ gives a complete list of inequivalent irreducible $SF(d)^+$ modules.
\end{thm}
Abe also exhibits reducible and indecomposable extensions of $SF(d)^\pm$
denoted by $\widehat{SF(d)}^\pm$. From the existence of such modules it follows that $SF(d)^+$ can't be rational.

To get the epimorphism $\pi : A(SF(d)^+) \to \mathcal A_d$ and to prove Proposition \ref{propinv}, we will need to recall the action of $A(SF(d)^+)$
on top components of these representations, described in \cite{Abe}.
By a result from \cite{Zhu}, the images of strong generators of $SF(d)^+$ are the generators of
$A(SF(d)^+)$, so we need to consider only the action of strong generators.  

\subsubsection{$SF(d)_{\theta}^+$}
The top component of $SF(d)_{\theta}^+$ is one-dimensional, spanned
with the vacuum vector $\mathbf 1_{\theta}$ with conformal weight $-d/8$.
As for the action of the strong generators of $SF(d)^+$, we have
$$o(h^{ii})\mathbf 1_{\theta} = - \frac{1}{8} \mathbf 1_{\theta}, $$
while the rest of the generators act as $0$.

\subsubsection{$SF(d)_{\theta}^-$}
The top component of $SF(d)_{\theta}^-$ is $2d$-dimensional, spanned by
$e_{-\frac{1}{2}}^k \mathbf 1_{\theta}, f_{-\frac{1}{2}}^k \mathbf 1_{\theta},$
for $1 \leq k \leq d$, with conformal weight $-d/8 + 1/2$.
We will slightly abuse notation and write
$$e^k = e_{-\frac{1}{2}}^k \mathbf 1_{\theta}, f^k = f_{-\frac{1}{2}}^k \mathbf 1_{\theta},$$
for simplicity.
Let $1 \leq i, j \leq d, i \neq j$. The small generators act as:
\begin{align*}
o(e^{ij})e^k &= 0, &o(e^{ij})f^k &= \frac{1}{2}(\delta_{ik}e^j - \delta_{jk}e^i) \\
o(f^{ij})e^k &= \frac{1}{2}(\delta_{jk}f^i - \delta_{i,k}f^j), &o(f^{ij})f^k &= 0 \\ 
o(h^{ij})e^k &= \frac{1}{2}\delta_{jk}e^i, &o(h^{ij})f^k &= \frac{1}{2}\delta_{ik}f^j, \ (i \neq j)\\
o(h^{ii})e^k &= \frac{1}{2}\delta_{ik}e^i - \frac{1}{8}e^k, &o(h^{i,i})f^k &= \frac{1}{2}\delta_{ik}f^i - \frac{1}{8}f^k.
\end{align*}

The large generators act as (now we allow $i = j$):
\begin{align*}
o(E^{ij})e^k &= 0, &o(E^{ij})f^k &= - \frac{1}{4} (\delta_{ik}e^j + \delta_{jk}e^i) \\
o(F^{ij})e^k &= \frac{1}{4} (\delta_{jk}f^i + \delta_{ik}f^j), &o(F^{ij})f^k &= 0 \\ 
o(H^{ij})e^k &= \frac{1}{4} \delta_{jk}e^i, &o(H^{ij})f^k &= - \frac{1}{4} \delta_{ik}f^j.
\end{align*}.

\subsubsection{$SF(d)^+$}
The top component of $SF(d)^+$ is $1$-dimensional (spanned
by $\mathbf 1$) of conformal weight $0$.
All generators act trivially.

\subsubsection{$SF(d)^-$}
The top component of $SF(d)^-$ is $2d$-dimensional, spanned by
$e^k, f^k, 1 \leq k \leq d$ with conformal weight $1$.
The small generators act as
\begin{align*}
o(e^{ij})e^k &= 0, &o(e^{ij})f^k &= \delta_{ik}e^j - \delta_{jk}e^i \\
o(f^{ij})e^k &= \delta_{jk}f^i - \delta_{ik}f^j, &o(f^{ij})f^k &= 0 \\ 
o(h^{ij})e^k &= \delta_{jk}e^i, &o(h^{ij})f^k &= \delta_{ik}f^j,
\end{align*}
while the large generators act as
\begin{align*}
o(E^{ij})e^k &= 0, &o(E^{ij})f^k &= - \delta_{ik}e^j - \delta_{jk}e^i \\
o(F^{ij})e^k &= \delta_{jk}f^i + \delta_{ik}f^j, &o(F^{ij})f^k &= 0 \\ 
o(H^{ij})e^k &= \delta_{jk}e^i, &o(H^{ij})f^k &= - \delta_{ik}f^j.
\end{align*}

\subsubsection{$\widehat{SF(d)}^+$}
In the Chapter 5 of \cite{Abe}, it was shown that $\Omega(\widehat{SF(d)}^+)
= \widehat{SF(d)}^+_0$.
The top component $\widehat{SF(d)}^+_0$ is $2^{2d-1}$-dimensional 
and isomorphic to the even part of an exterior algebra on generators
$e^i_0, f^i_0, 1 \leq i \leq d$. Its conformal weight is $0$.
Small generators act as multiplication by monomials of length 2 in
the following way:
$$o(e^{ij}) = e^i_0 e^j_0, \ o(h^{ij}) = e^i_0 f^j_0, \ o(f^{ij}) = f^i_0 f^j, $$
while large generators act as $0$.
Notice that $o(\omega)$ is a nilpotent operator of degree $d+1$.

\subsubsection{$\widehat{SF(d)}^-$}
In the Chapter 5 of \cite{Abe} it was shown that
$$\Omega(\widehat{SF(d)}^-) = \widehat{SF(d)}^-_{(0)} \oplus \widehat{SF(d)}[2d]^-_{(1)},$$
where $\widehat{SF(d)}[2d]^-_{(1)}$ is isomorphic as an $A(SF(d)^+)$-module
to the top component of $SF(d)^-$.

\subsection{Epimorphism $\pi : A(SF(d)^+) \to \mathcal A_d$} \label{subsec:epi}
In this subsection we will show the existence of an algebra epimorphism
$\pi : A(SF(d)^+) \to \mathcal A_d$.
For this, we will use a following generalization of the Chinese remainder theorem (see \cite{IR}):
\begin{thm}
  \label{KTOO}
  Let $R$ be a ring (with unity), and let $I_1, I_2, \ldots, I_k$ be
  two-sided ideals in $R$. If these ideals are pairwise coprime, then
  there is an epimorphism
  \begin{align*}
    \pi : R &\to R/I_1 \oplus R/I_2 \oplus \ldots \oplus R/I_k \\
          x &\mapsto (x + I_1, x + I_2, \ldots, x + I_k)
  \end{align*}
  with $\ker \pi = I_1 \cap I_2 \cap \ldots \cap I_k$.
\end{thm}

We will denote the homomorphisms connected to the representations of $A(SF(d)^+)$
in the following way:
\begin{align*}
  \rho^\pm &: A(SF(d)^+) \to \End_{\mathbb C} (\Omega(SF(d)^\pm)) \\
  \rho_\theta^\pm &: A(SF(d)^+) \to \End_{\mathbb C} (\Omega(SF_\theta(d)^\pm)) \\
  \widehat{\rho^\pm} &: A(SF(d)^+) \to \End_{\mathbb C} (\widehat{\Omega(SF(d)^\pm)}).
\end{align*}

We want to use Theorem \ref{KTOO} on the ideals $\ker(\widehat{\rho^+}), \ker(\rho^-),\ker(\rho_\theta^\pm)$ -
for that, we need to check that those ideals are
pairwise coprime and we need to know the images of these homomorphisms.
We get the coprimality by observing the minimal polynomials of
$o(\omega)$ on different representations.
We have:
\begin{align*}
  [\omega]^{d+1} &\in \ker(\widehat{\rho^+}) \\
  [\omega] - 1 &\in \ker(\rho^-)\\
  [\omega] + d/8  &\in \ker(\rho_\theta^+)\\
  [\omega] + (d/8-1/2) &\in \ker(\rho_\theta^-).
\end{align*}
Notice that all these polynomials in $[\omega]$ are pairwise coprime.

By Jacobson density theorem we get
$$\im(\rho^-) \simeq \im(\rho^-_\theta) \simeq M_{2d}(\mathbb C), \ \im(\rho^+_\theta) \simeq \mathbb C$$
for irreducible representations.
It's clear from the previous subsection that $\im(\widehat{\rho^+})$ is
isomorphic to $\Lambda^{\text{even}}(V_{2d})$ (even part of exterior algebra on
the standard representation of $\mathfrak{sp}(2d)$).
Now, we can use Theorem \ref{KTOO} to get an epimorphism
$$\pi : A(SF(d)^+) \to \mathcal A_d = M_{2d}(\mathbb C) \oplus M_{2d}(\mathbb C) \oplus \Lambda^{\text{even}}(V_{2d}) \oplus \mathbb C. $$
\newpage
\section{Spanning set for $\mathcal P(SF(d)^+)$}
The main result of this section is the existence of a spanning set for
$\mathcal P(SF(d)^+)$ of cardinality $n_d = 2^{2d-1} + 8d^2 + 1$.
Let us describe it more precisely.
Put $x^i = e^i, x^{i+d} = f^i$, for $1 \leq i \leq d$.
Denote $B_d^1$ the set of all even length monomials obtained by $(-1)$th
multiplication of $x^i$, and $B_d^2$ the set of the following vectors:
\begin{align*}
   x^i_{-2}x^j,  &1 \leq i \leq j \leq 2d \\
   x^i_{-3}x^j,  &1 \leq i < j \leq 2d \\
   x^i_{-4}x^j,  &1 \leq i \leq j \leq 2d \\
   x^i_{-5}x^j,  &1 \leq i < j \leq 2d \\
   e^1_{-7}f^1. &\ 
\end{align*}
It is easy to see that $|B_d^1| = 2^{2d-1}$ and $|B_d^2| = 8d^2 + 1$.
The main Theorem of this section is:
\begin{thm} \label{Bd}
  Image of $B_d = B_d^1 \cup B_d^2$ in $\mathcal P(SF(d)^+)$ is
  a spanning set of $\mathcal P(SF(d)^+)$.
\end{thm}

Also, in this section we will write $C_2 := C_2(SF(d)^+)$ for simplicity.
\subsection{Some relations in $\mathcal P(SF(d)^+)$} \label{subsec:lem}
In this section we prove some lemmas that will be helpful in
the proof of Theorem \ref{Bd}. First, let us recall some definitions
and relations from \cite{Abe}.
In the Chapter 3 of \cite{Abe} it is defined that:
\begin{equation}
  \label{eq:Bdef}
  B_{m,n}(a,b) = \frac{(m-1)! (n-1)!}{(m+n-1)!} a_{-m}b_{-n}\mathbf 1,
\end{equation}
for $a, b \in \mathfrak h$.
Following relations hold in $\mathcal P(SF(d)^+)$:
\begin{align}
  \overline{B_{m,n}(a,b)} &= (-1)^{n-1} \overline{B_{m+n-1, 1}(a,b)} \label{bmn1}\\
  \overline{B_{m,n}(a,b)} &= (-1)^{m+n-1}\overline{B_{m,n}(b,a)}. \label{bmn2}
\end{align}
These relations show that every monomial of length 2 in $SF(d)^+$ can
be written modulo $C_2$ as a linear combination of the vectors
\begin{align*}
  x^i_{-n}x^j, &\text{ where } n \in \mathbb Z_{>0} \text{ is odd and }   1 \leq i < j \leq 2d, \\
  x^i_{-n}x^j, &\text{ where } n \in \mathbb Z_{>0} \text{ is even and }   1 \leq i \leq j \leq 2d.
\end{align*}

Also, we will need some relations from the proof of Proposition 3.12 in \cite{Abe}:
\begin{align}
  (\overline{e^i_{-m}f^i}) \cdot (\overline{e^i_{-k}f^i}) &= mk \left(\frac{1}{m+1} + \frac{1}{k+1}\right) {m + k \choose k} \overline{e^i_{-m-k-1}f^i} 
                                    \label{efef} \\
  \overline{e^i_{-6}e^i} &= \overline{0} \label{e6}\\ 
  \overline{e^i_{-9}f^i} &= \overline{0} \label{e9}
\end{align}

Now, we are ready to prove an important Lemma.

\begin{lem} \label{lemlem}
  In $\mathcal P(SF(d)^+)$ we have
  \begin{align}
    (\overline{h^{ii} - h^{jj}})^3 &= \overline{0} \label{eq:l1}\\
    (\overline{h^{ii} + h^{jj}})^3 &= 12 (\overline{h^{ii} + h^{jj}}) \cdot \overline{h^{ii}} \cdot \overline{h^{jj}}. \label{eq:l2} 
  \end{align}
  for $1 \leq i < j \leq d$.
\end{lem}
\begin{proof}
  Because of the permutation automorphisms (cf. subsection \ref{subsec:auto}, it is sufficient to show this in case
  $i = 1, j = 2$ (in both cases). Relation \eqref{eq:l1} will follow from:
  \begin{align}
    (\overline{h^{11} - h^{22}}) \cdot \overline{h^{12}} &= \overline{0}
                                                           \label{lem1rel1}\\
    \overline{h^{12}} \cdot \overline{h^{21}} &= \frac 1 2 (\overline{h^{11} - h^{22}}) \label{lem1rel2}
  \end{align}
  We calculate:
  \begin{align*}
    h^{11}_{-1}h^{12} &= e^1_{-3}f^1_1 e^1_{-1}f^2 \\
                      &= e^1_{-3}f^2 \\
    h^{22}_{-1}h^{12} &= - f^2_{-3}e^2_1 e^1_{-1}f^2 \\
                      &= - f^2_{-3}e^1 \\
                      &\stackrel{\eqref{bmn2}}{=} e^1_{-3}f^2 \mod C_2.
  \end{align*}
  The second one follows from
  $$(\overline{h^{ii}})^2 \stackrel{\eqref{efef}}{=} 2 \overline{e^i_{-3}f^i} $$
  and some direct calculation.

  Now we turn to the relation \eqref{eq:l2}.
  Let's denote
  $$Z:= (h^{11})_{-1}h^{22} = e^1_{-1}f^1_{-1}e^2_{-1}f^2.$$
  Then, we need to prove
  $$\overline{\omega}^3 = 12 \overline{\omega} \overline{Z}.$$
  We check that
  $$L_{-2}Z = \sum_{sym} e^1_{-3}f^1_{-1}e^2_{-1}f^2_{-1}\mathbf 1, $$
  and calculate:
  \begin{align*}
    h^{11}_{-3}h^{22} + h^{22}_{-3}h^{11} &= L_{-2}Z + e^1_{-2}f^1_{-2}e^2_{-1}f^2_{-1}\mathbf 1 + e^1_{-1}f^1_{-1}e^2_{-2}f^2_{-2}\mathbf 1 \\
    e^{12}_{-3}f^{12} + f^{12}_{-3}e^{12} &= -L_{-2}Z - e^1_{-2}f^1_{-1}e^2_{-2}f^2_{-1}\mathbf 1 - e^1_{-1}f^1_{-2}e^2_{-1}f^2_{-2}\mathbf 1 \\ &\  + \sum (e^i_{-5}f^i - f^i_{-5} e^i) \\
    h^{12}_{-3}h^{21} + h^{21}_{-3}h^{12} &= -L_{-2}Z - e^1_{-2}f^1_{-1}e^2_{-1}f^2_{-2}\mathbf 1 - e^1_{-1}f^1_{-2}e^2_{-2}f^2_{-1}\mathbf 1 \\ &\  + \sum (e^i_{-5}f^i - f^i_{-5} e^i). \\
  \end{align*}

  If we subtract the second and third equation from the first one we see
  $$3 \overline{\omega} \cdot \overline{Z} + \sum_{sym} \overline{e^1_{-2} e^2_{-2} f^1_{-1} f^2_{-1}\mathbf 1} = 2 \sum (\overline{e^i_{-5}f^i - f^i_{-5} e^i})$$
Because of
  $$L_{-1}^2 Z = 2L_{-2}Z + 2 \sum_{sym} e^1_{-2} e^2_{-2} f^1_{-1} f^2_{-1}\mathbf 1,$$
  we can simplify to
  $$\overline{\omega} \cdot \overline{Z} = \sum (\overline{e^i_{-5}f^i - f^i_{-5} e^i}) \stackrel{\eqref{bmn2}}{=} 2 \sum(\overline{e^i_{-5}f^i}).$$
   Using the relation \eqref{efef} we have:
    $$\overline{\omega} \overline{Z}
    = \frac{1}{9}( \overline{h^{11}}^3 + \overline{h^{22}}^3).$$
To conclude:
  \begin{align*}
    \overline{\omega}^3 &= \overline{h^{11}}^3 +3 \overline{h^{11}}^2\overline{h^{22}} + 3\overline{h^{11}}\overline{h^{22}}^2 + \overline{h^{22}}^3 \\
                        &= \overline{h^{11}}^3 + 3 \overline{\omega} \overline{Z} + \overline{h^{22}}^3 \\
    &= 12 \overline{\omega} \overline{Z}.
  \end{align*}
\end{proof}

\begin{cor} \label{cor}
  In $\mathcal P(SF(d)^+)$ for $1 \leq i, j \leq d$ we have
  $$(\overline{h^{ii}})^4 = (\overline{h^{jj}})^4,$$
  which is equivalent to
  $$\overline{e^i_{-7}f^i} = \overline{e^j_{-7}f^j}.$$
\end{cor}
\begin{proof}
  Equivalence follows from the formula \eqref{efef}.
  Let's put $z_i:= \overline{h^{ii}}$. In that notation, previous two
  Lemmas say:
  \begin{align*}
    p(z_i, z_j) &:= (z_i - z_j)^3 = \overline{0} \\
    q(z_i, z_j) &:= (z_i + z_j)^3 - 12(z_i + z_j)z_i z_j = \overline{0}.
  \end{align*}
  Now, notice that we have:
  $$z_i^4 - z_j^4 = \frac{1}{4} (5(z_i + z_j)p(z_i, z_j) + (z_j - z_i)q(z_i, z_j) ). $$
\end{proof}
\subsection{Monomials of length 2}

The goal of this subsection is to prove the following:
\begin{prop} \label{propd2}
  All monomials of length 2 in $SF(d)^+$ can be written (modulo $C_2$)
  as a linear combination of the following vectors:
  \begin{align*}
    x^i_{-1}x^j,  &1 \leq i < j \leq 2d \\
   x^i_{-2}x^j,  &1 \leq i \leq j \leq 2d \\
   x^i_{-3}x^j,  &1 \leq i < j \leq 2d \\
   x^i_{-4}x^j,  &1 \leq i \leq j \leq 2d \\
   x^i_{-5}x^j,  &1 \leq i < j \leq 2d \\
   e^1_{-7}f^1. &\ 
  \end{align*}
 \end{prop}
Notice that the first row lies in $B_d^1$ and that other rows lie in $B_d^2$.
As we already noticed in the subsection \ref{subsec:lem}, relations \eqref{bmn1} and \eqref{bmn2} imply that every monomial of length 2 in $SF(d)^+$ can
be written (modulo $C_2$) as a linear combination of vectors
\begin{align*}
  x^i_{-n}x^j, &\text{ where } n \in \mathbb Z_{>0} \text{ is odd and }   1 \leq i < j \leq 2d, \\
  x^i_{-n}x^j, &\text{ where } n \in \mathbb Z_{>0} \text{ is even and }   1 \leq i \leq j \leq 2d.
\end{align*}

We need to remove the monomials with higher $n$ from the list.

\begin{lem}
  For $1 \leq i \leq d$ we have:
  \begin{align*}
    e^i_{-k}e^i &\equiv 0 \mod C_2, \text{ for } k \geq 6 \\
    e^i_{-k}f^i &\equiv 0 \mod C_2, \text{ for } k = 6 \text{ or } k \geq 8.
  \end{align*}
\end{lem}
\begin{proof}
  First, using the relation \eqref{bmn1} we get:
  
    \begin{align*}
    h^{ii}_{-1}(e^i_{-k}e^i) &= k e^i_{-k-2}e^i + e^i_{-k}e^i_{-3}\mathbf 1 \\
                             &\equiv \left(k + {k + 1 \choose 2} \right) e^i_{-k-2}e^i \mod C_2\\
    h^{ii}_{-1}(e^i_{-k}f^i) &= k e^i_{-k-2}f^i + e^i_{-k}f^i_{-3}\mathbf 1 \\
                             &\equiv \left(k + {k + 1 \choose 2} \right) e^i_{-k-2}f^i \mod C_2
  \end{align*}

  It follows that
  \begin{align}
    e^i_{-k}e^i \in C_2 &\implies e^i_{-k-2}e^i \in C_2 \label{eq:m21}\\
    e^i_{-k}f^i \in C_2 &\implies e^i_{-k-2}f^i \in C_2 \label{eq:m22}
  \end{align}
  From relation \eqref{bmn2} we know that $e^i_{-2k-1}e^i \in C_2, k \geq 0$,
  and we have a relation \eqref{e6} which says that $e^i_{-6}e^i \in C_2$.
  From \eqref{eq:m21} it follows that $e^i_{-k}e^i \in C_2$, for $k \geq 6$.

  We can use $V_i \in \mathfrak{sp}(2d)$ to get
  \begin{align*}
    C_2 \ni V_i(e^i_{-6}e^i) &= f^i_{-6}e^i + e^i_{-6}f^i \\
                   & \stackrel{\eqref{bmn2}}{\equiv} 2e^i_{-6}f^i \mod C_2,
  \end{align*}
  and then proceed in similar fashion (together with the relation \eqref{e9}).
\end{proof}

Using $\mathfrak{sp}(2d)$ on the relations from the previous Lemma gives us
everything but the last row in Proposition \ref{propd2}.
For the last row, we use the Corollary \ref{cor}. 

\subsection{Monomials of length greater than 2}
Following \cite{Abe}, we define $\mathcal L^0 SF(d) = \mathbb C \mathbf 1$, and for $r \in \mathbb Z_{> 0}$
$$\mathcal L^r SF(d) = \sspan_\mathbb C \{x^{i_1}_{-n_1} \ldots x^{i_s}_{-n_s} \mathbf 1 : 1 \leq i_j \leq 2d, n_j \in \mathbb Z_{> 0}, s \leq r\}. $$
We also put $\mathcal L^r SF(d)^+ = \mathcal L^r SF(d) \cap SF(d)^+$.
In this section, we will only use $\mathcal L^r SF(d)^+$, so we will
write $\mathcal L^{r+} := \mathcal L^r SF(d)^+$ for simplicity.

We can generalize Abe's definition \eqref{eq:Bdef} to
\begin{equation}
  B_{m_1, \ldots, m_k}(a^1, \ldots, a^k) := \frac{(m_1 - 1)! (m_2 - 1)! \ldots (m_k - 1)! }{(\sum m_i - 1)!} a^1_{-m_1} a^2_{-m_2} \ldots a^k_{-m_k} \mathbf 1.
\end{equation}
Similar to $k=2$ case, we have
$$L_{-1}B_{m_1, \ldots, m_k}(a_1, \ldots, a_k) = (\sum_{i = 1}^k m_i) \cdot \sum_{i = 1}^k B_{m_1, \ldots, m_i + 1, \ldots, m_k}(a_1, \ldots, a_k) \in C_2$$
Iterating this relation, we get the following Lemma:
\begin{lem}\label{lemM1}
  Each element of $SF^+(d)$ can be written (modulo $C_2$)
  as a linear combination of monomials of the form
  $$x^{i_1}_{-n_1} \ldots x^{i_{2k}}_{-n_{2k}}\mathbf 1, k \geq 0, n_j > 0, 1 \leq i_j \leq 2d, $$
  where at least one of $n_j$ equals $1$.
\end{lem}

We can improve on this Lemma:
\begin{lem}\label{lemM2}
  Each element of $SF^+(d)$ can be written (modulo $C_2$)
  as a linear combination of monomials of the form
  \begin{equation} \label{form}
  x^{i_1}_{-n}x^{i_2}_{-1} \ldots x^{i_{2k}}_{-1}\mathbf 1, k \geq  0, n > 0, 1 \leq i_j \leq 2d  
  \end{equation}
  
\end{lem}
\begin{proof}
  We use the induction on the monomial length $2k$. For $k = 0$ the result is obvious, and for $k=1$ it follows from the Proposition \ref{propd2}.
  Let $k \geq 2$: \\
  Because of the Lemma \ref{lemM1} it is enough to consider the monomials
  of the form
  $$a = x^{i_1}_{-n_1} \ldots x^{i_{2k-1}}_{-n_{2k-1}} x^{i_{2k}}.$$
  We calculate, using the associativity formula (see \cite{LL}):
  \begin{align*}
    &(x^{i_{2k-1}}_{-n_{2k-1}}x^{i_{2k}})_{-1} (x^{i_1}_{-n_1} \ldots x^{i_{2k-2}}_{-n_{2k-2}} \mathbf 1) \\
    &= \sum_{j \geq 0} (-1)^j {-n_{2k-1} \choose j} \left(x^{i_{2k-1}}_{-n_{2k-1} - j}x^{i_{2k}}_{-1 + j} + (-1)^{-n_{2k-1}} x^{i_{2k}}_{-n_{2k-1} -1 - j}x^{i_{2k-1}}_{j} \right) (x^{i_1}_{-n_1} \ldots x^{i_{2k-2}}_{-n_{2k-2}} \mathbf 1) \\
    &= x^{i_{2k-1}}_{-n_{2k-1}}x^{i_{2k}}_{-1}x^{i_1}_{-n_1} \ldots x^{i_{2k-2}}_{-n_{2k-2}} \mathbf 1 + u, u \in \mathcal L^{(2k-2)+} \\
  &= a + u, u \in \mathcal L^{(2k-2)+}.
  \end{align*}

  Now, we can use the inductive hypothesis on $(x^{i_1}_{-n_1} \ldots x^{i_{2k-2}}_{-n_{2k-2}} \mathbf 1)$ and the fact that ($-1$)-multiplication preserves
  the $C_2$ space to conclude that we can restrict ourselves to the
  monomials of the form
  $$b = x^{i_1}_{-n_1}x^{i_2}_{-n_2}x^{i_3}_{-1}x^{i_4}_{-1} \ldots x^{i_{2k-1}}_{-1} x^{i_{2k}}.$$
  We can use a similar trick to write
  $$b = (x^{i_{2k-1}}_{-1}x^{i_{2k}})_{-1} (x^{i_1}_{-n_1}x^{i_2}_{-n_2} x^{i_3}_{-1} \ldots x^{i_{2k-2}}) + u, u \in \mathcal L^{(2k-2)+} $$
  and use the inductive hypothesis on $(x^{i_1}_{-n_1}x^{i_2}_{-n_2} x^{i_3}_{-1} \ldots x^{i_{2k-2}})$, proving our result.
\end{proof}

Next step is to check that the monomials of the form \eqref{form} reside in $\sspan_{\mathbb C} B_d + C_2$.
We will first have to prove this for monomials of length $4$ and then proceed with an inductive argument.

\begin{lem}[Technical Lemma] \label{lemd4}
  Monomials of length 4 are
  in $\sspan_{\mathbb C} B_d + C_2$.
\end{lem}
\begin{proof}
  See subsection \ref{subsec:tl}.
\end{proof}

Now, we can finish the proof of Theorem \ref{Bd}.

\begin{prop} \label{fprop}
  All monomials in $SF(d)^+$ of length greater or equal to 4 lie in
  $$\sspan_{\mathbb C} B_d + C_2.$$
\end{prop}
\begin{proof}
  We prove this Proposition by the induction on the length of monomials.
  The base case (monomials of length 4) is proven in previous Lemmas.
  Now, assume that all monomials of length less than $2k$ lie in
  $\sspan_{\mathbb C} B_d + C_2$.
  Because of the Lemma \ref{lemM2} we can reduce to monomials of
  the form
  $$ x^{i_1}_{-n}x^{i_2}_{-1} \ldots x^{i_{2k}}_{-1}\mathbf 1, n > 0, 1 \leq i_j \leq 2d. $$
  We can write
  $$x^{i_1}_{-n}x^{i_2}_{-1} \ldots x^{i_{2k}}_{-1}\mathbf 1 = (x^{i_{2k-1}}_{-1}x^{i_{2k}})_{-1} (x^{i_1}_{-n}x^{i_2}_{-1} \ldots x^{i_{2k-2}}_{-1}\mathbf 1) + u, u \in \mathcal L^{(2k-2)+}.$$
  By the inductive hypothesis,
  $$x = x^{i_1}_{-n}x^{i_2}_{-1} \ldots x^{i_{2k-2}}_{-1}\mathbf 1 \in \sspan_{\mathbb C} B_d + C_2.$$
  This means that we can write $x$ (modulo $C_2$) as a linear combination of monomials
  from $B_d$ of length less than $2k.$
  Let $m_1 \in B_d^1$ of length $<2k$. Then, $(x^{i_{2k-1}}_{-1}x^{i_{2k}})_{-1}m_1$
  will be a linear combination of a monomial from $B_d^1$ of length $2k$ and some monomials
  of length $< 2k$.

  Let $m_2 \in B_d^2$. Then $(x^{i_{2k-1}}_{-1}x^{i_{2k}})_{-1}m_2$ will be
  a linear combination of monomials of length $\leq 4$, and that case is
  already covered.
\end{proof}

\subsection{Proof of the Technical Lemma} \label{subsec:tl}

In this section we prove the Technical Lemma \ref{lemd4}, that is,
all monomials of length $4$ lie in $\sspan_{\mathbb C} B_d + C_2$.

Recall the automorphisms $\tau_i$ and the permutation automorphisms
from subsection \ref{subsec:auto}.
 If we take a monomial of the form \eqref{form} of length 4, we can use
them to write it as $e_{-n}^1 x^{i_1}_{-1}x^{i_2}_{-1}x^{i_3}.$
We split this problem into several cases, depending on how many
of $i_1, i_2, i_3$ equal $1$.
We immediately see that we can't have a non-zero monomial
where $i_1 = i_2 = i_3 = 1$. So, we start with:

  \paragraph{\textbf{Case 1}: Exactly two of $i_1, i_2, i_3$ equal $1$} \mbox{}\newline
  In this case, without loss of generality, we can take that
  monomial is of the form $e_{-n}^1 e^1_{-1}f^1_{-1}x^{i_3}$, where $i_3 \neq 1$.
  Using the automorphisms, it is enough to consider the monomial
  $e_{-n}^1 e^1_{-1}f^1_{-1}e^2$.
  Notice that
  \begin{align*}
    e^{12}_{-2}h^{11} &= e^1_{-2}e^2_{-1}e^1_{-1}f^1 + u, u \in \mathcal L^{2+} \\
    e^{12}_{-3}h^{11} &= (e^1_{-3}e^2_{-1} + e^1_{-2}e^2_{-2})e^1_{-1}f^1 + v, v \in \mathcal L^{2+}
  \end{align*}
  The first row shows that  $e_{-2}^1 e^1_{-1}f^1_{-1}e^2 \in C_2 + L^{2+}$,
  and we know that $\mathcal L^{2+} \subseteq \sspan_{\mathbb C}B_d + C_2$ by Proposition \ref{propd2}.
  Because of
  $$h^{22}_0 e^1_{-2}e^2_{-1}e^1_{-1}f^1 = e^1_{-2}e^2_{-2}e^1_{-1}f^1$$
  we can conclude from the second row that $e^1_{-3}e^2_{-1}e^1_{-1}f^1 \in \sspan_{\mathbb C} B_d + C_2.$
 It is easy to see we can continue using this trick for higher $n$.
  \paragraph{Case 2: Exactly one of $i_1, i_2, i_3$ equals $1$} \mbox{}\\
  Without loss of generality, we can take the monomial of
  the form $e_{-n}^1 e^1_{-1}x^{i_2}_{-1}x^{i_3}$ or $e_{-n}^1 f^1_{-1}x^{i_2}_{-1}x^{i_3}$
  for $i_2, i_3 \neq 1$.
  There is also a question of whether or not $i_2$ equals $i_3$.
  Using automorphisms, we can reduce to these four subcases:
  \paragraph{Subcase 2.1: $e_{-n}^1e^1_{-1}e^2_{-1}e^3$} \mbox{}\\
  We use a similar idea as in \textbf{Case 1}:
  \begin{align*}
    e^{13}_{-2}e^{12} &= e^1_{-2}e^3_{-1}e^1_{-1}e^2_{-1} \\
    e^{13}_{-3}e^{12} &= (e^1_{-3}e^3_{-1} + e^1_{-2}e^3_{-2})e^1_{-1}e^2_{-1}.
  \end{align*}
  The first row shows that $e^1_{-2}e^3_{-1}e^1_{-1}e^2_{-1} \in C_2$,
  and because of
  $$h^{33}_0  e^1_{-2}e^3_{-1}e^1_{-1}e^2_{-1} = e^1_{-2}e^3_{-2}e^1_{-1}e^2_{-1} \in
  C_2$$
  the second row shows that $e^1_{-3}e^3_{-1}e^1_{-1}e^2_{-1} \in C_2$.
  We can continue using the same trick for higher $n$.
  \paragraph{Subcase 2.2: $e_{-n}^1f^1_{-1}e^2_{-1}e^3$} \mbox{}\\
  Looking at $e^{23}_{-2}(e^1_{-1}f^1), e^{12}_{-2}(e^3_{-1}f^1), e^{13}_{-2}(e^2 f^1)$ we can conclude that
  \begin{equation} \label{elem}
  e^1_{-2}f^1_{-1}e^2_{-1}e^3, e^1_{-1}f^1_{-1}e^2_{-2}e^3 \in C_2 + \mathcal L^{2+} \subseteq C_2 + \sspan_{\mathbb C} B_d.  
  \end{equation}
  Now if we calculate
  $$e^{12}_{-3}(e^{3}_{-1}f^1) = (e^1_{-1}e^2_{-3} + e^1_{-2}e^2_{-2} + e^1_{-3}e^2_{-1})e^{3}_{-1}f^1 + u, u \in \mathcal L^{2+}, $$
  and use $h^{22}_0$ on the relation \eqref{elem}, we get
  $e^1_{-3}f^1_{-1}e^2_{-1}e^3 \in C_2 + \sspan_{\mathbb c} B_d$.
  We can use the same strategy for higher $n$.
  \paragraph{Subcase 2.3: $e_{-n}^1e^1_{-1}e^2_{-1}f^2$} \mbox{}\\
  From the proof of the Proposition \ref{propd2} we know that
  $e^1_{-n}e^1 \in C_2$ for all $n \in \mathbb N$ except
  $n = 2, 4$. Because of
  $$h^{22}_{-1}(e^1_{-n}e^1) = e_{-n}^1e^1_{-1}e^2_{-1}f^2$$
  we only need to check for these $n$.
  For $n = 2$ we have
  $$e^{12}_{-2}(e^1_{-1}f^2) = e^1_{-2}e^2_{-1}e^1_{-1}f^2 + u, u \in \mathcal L^{2+}.$$
  For $n = 4$ we need to calculate:
    \begin{align*}
    H^{11}_0(e^1_{-2}e^1e^2f^2) = 3e^1_{-4}e^1e^2f^2 + 2e^1_{-2}e^1_{-3}e^2f^2 \\
    h^{11}_0(e^1_{-3}e^1e^2f^2) = 3e^1_{-4}e^1e^2f^2 + e^1_{-3}e^1_{-2}e^2f^2.
  \end{align*}
  \paragraph{Subcase 2.4: $e_{-n}^1f^1_{-1}e^2_{-1}f^2$} \mbox{}\\
  From the proof of the Proposition \ref{propd2} we know that
  $$e^1_{-n}f^1 \in C_2, \forall n \in \mathbb N, n \neq 1,2,3,4,5,7. $$
  Because of $e_{-n}^1f^1_{-1}e^2_{-1}f^2 = h^{22}_{-1}(e_{-n}f^1)$ we can
  reduce to cases $n = 2,3,4,5,7$.
  For the $n = 2, 4$ cases we can use the \textbf{subcase 2.3}
  together with the fact that for $V_1 \in \mathfrak{sp}(2d)$:
  $$V_1(e_{-n}^1e^1_{-1}e^2_{-1}f^2) = f_{-n}^1e^1_{-1}e^2_{-1}f^2 + e_{-n}^1f^1_{-1}e^2_{-1}f^2 \stackrel{\eqref{bmn2}}{\equiv} 2 e_{-n}f^1_{-1}e^2_{-1}f^2 \mod C_2$$
  For the $n = 3,5,7$ we can use the relation \eqref{efef}
  to show that, up to scalars, $e_{-n}^1f^1_{-1}e^2_{-1}f^2$ is equivalent
  to $(\overline{h^{11}})^{\frac{n+1}{2}} \cdot \overline{h^{22}}$
  for odd $n$.
  We can use the relations \eqref{bmn1} and \eqref{bmn2} (similar to
  the proof of Corollary \ref{cor}) to show they are in
  $C_2 + \mathcal L^{2+}$.
  \paragraph{Case 3: None of $i_1, i_2, i_3$ equal $1$}\mbox{}  \\
  In the case all $i_1, i_2, i_3$ are different, by use of the automorphisms
  we can see that it's enough to consider the monomials $e^1_{-n}e^2_{-1}e^3_{-1}e^4_{-1}$.
  If we look at $e^{ij}_{-2}e^{kl}$ for different choices of $i,j,k,l \in \{1,2,3,4\}$,
  we get $e^1_{-2}e^2_{-1}e^3_{-1}e^4_{-1} \in C_2{SF(d)^+}.$
  We can proceed using the relation
  $$h^{11}_0 (e^1_{-n}e^2_{-1}e^3_{-1}e^4_{-1}) =  n e^1_{-n-1}e^2_{-1}e^3_{-1}e^4_{-1}.$$

  If some of $i_1, i_2, i_3$ take the same value, we get the monomial of the form
  $e^1_{-n}e^i_{-1}f^i_{-1}e^j$ or $e^1_{-n}e^i_{-1}f^i_{-1}f^j$ for $i \neq j$.
  Now we can use the permutation automorphism which exchanges $1$ and $i$
  to put ourselves in one of the former cases.

\newpage
\section{Consequences}

\subsection{Dimension of the space of one-point functions on $SF(d)^+$}
Following \cite{AN}, denote by $\mathcal C(V)$ a vector space of
one-point functions on a VOA $V$.
In \cite{AN} (Theorem 6.3.2), the authors have constructed $2^{2d-1} + 3$ linearly independent pseudo-trace functions, therefore showing that
$$\dim_{\mathbb C} \mathcal C(SF(d)^+) \geq 2^{2d-1} + 3$$
holds for general $d$.
(We won't give the definitions of one-point functions or pseudo-trace
functions here, instead we refer the reader to \cite{Miy} and \cite{AN}.)

They also showed that the equality holds in the case of $d = 1$ (using
Abe's calculation of Zhu's algebra of $SF(1)^+$ from \cite{Abe}), and
conjectured it for general $d$.
Now, because of the Theorem \ref{uvod1}, we can use the same method to show that their conjecture holds.
That method depends on calculating the dimension of the space
of symmetric linear functions on the Zhu's algebra.

We recall the notion: let $A$ be a finite-dimensional associative
$\mathbb C$--algebra. A linear function $\varphi : A \to \mathbb C$
is called a symmetric linear function if
$$\varphi(ab) = \varphi(ba), \forall a, b \in A.$$
Following \cite{AN} we denote the space of those functions by
$S^A$, and for a VOA $V$, we put $S^V := S^{A(V)}$.
The main result connecting $S^V$ to $\mathcal C(V)$ is \cite[Theorem 3.3.6]{AN}:
\begin{thm}[\cite{AN}]
  Let $V$ be a $C_2$-cofinite VOA. Suppose that any simple $V$-module
  is infinite-dimensional.
  Then $\dim_{\mathbb C} \mathcal C(V) \leq \dim_{\mathbb C} S^V.$
\end{thm}
Now, all is set for the following Proposition.
\begin{prop}
  For all natural $d$, we have
  $$\dim_{\mathbb C}\mathcal C (SF(d)^+) = 2^{2d-1} + 3.$$
\end{prop}
\begin{proof}
  Using aforementioned results from \cite{AN}, it is enough to show that
  $\dim_{\mathbb C} S^{SF(d)^+} = 2^{2d-1} + 3$.
  By Theorem \ref{uvod1} we know that
  $$A(SF(d)^+) \simeq \Lambda^{\text{even}}(V_{2d}) \oplus M_{2d}(\mathbb C) \oplus M_{2d}(\mathbb C) \oplus \mathbb C.$$
  We need to know the dimension of the space of symmetric linear functions
  on each summand. 
  For the commutative algebra $\Lambda^{\text{even}}(V_{2d})$, each
  linear function is symmetric, giving us
  $$\dim_{\mathbb C}S^{\Lambda^{\text{even}}(V_{2d})} = \dim_{\mathbb C} \Lambda^{\text{even}}(V_{2d}) = 2^{2d-1}.$$
  For matrix algebras, we have
  $$\dim_{\mathbb C} S^{M_n(\mathbb C)} = 1, \ \forall n \in \mathbb N.$$
  The result follows.
\end{proof}

\subsection{Further properties of $A(SF(d)^+)$
  and $\mathcal P(SF(d)^+)$}

First, let's record some easy consequences of the Theorem \ref{uvod1}:

\begin{cor} \label{fcor}
  \begin{enumerate}
  \item Image of $B_d$ in $\mathcal P(SF(d)^+)$ is actually a basis for $\mathcal P(SF(d)^+)$.
  \item We have $\gr A(SF(d)^+) \simeq \mathcal P(SF(d)^+).$
  \item Minimal polynomial of $[\omega] \in A(SF(d)^+)$ is
    $$m_d(x) = x^{d+1} (x-1) \left(x + \frac{d}{8} \right) \left(x + \frac{d}{8}- \frac{1}{2} \right).$$
  \item The center of $A(SF(d)^+)$ is isomorphic to
    $$\Lambda^{\text{even}}(V_{2d}) \oplus \mathbb C \oplus \mathbb C \oplus \mathbb C. $$
  \end{enumerate}
\end{cor}
\begin{proof}
  Part 1 follows from Theorem \ref{Bd} and the equality of dimensions
  of $A(SF(d)^+)$ and $\mathcal P(SF(d)^+)$.
  Equality of dimensions also gives us that the epimorphism \eqref{eq:epi}
  is actually an isomorphism in this case, solving part 2.
  Part 3 is a direct consequence of $A(SF(d)^+) \simeq \mathcal A_d$
  and the discussion in the subsection \ref{subsec:epi}.
  Part 4 is a direct consequence of the Theorem \ref{uvod1}.
\end{proof}

Let $s_d$ be a degree of nilpotency of $\overline{\omega} \in \mathcal P(SF(d)^+)$. Parts 2 and 3 of the Corollary combine to give us that
$s_d \leq \deg m_d = d+4$. Actually, we can do better:

\begin{prop}
   We have:
  $$ s_d =
  \begin{cases}
    5, & \quad d \leq 4 \\
    d + 1, & \quad d \geq 5.
  \end{cases} $$
\end{prop}
\begin{proof}
  Notice that the element of maximal conformal weight in $B_d$
   is $e^1_{-7}f^1$ for $d \leq 4$, and $e^1_{-1}f^1_{-1}e^2_{-1}f^2_{-1}
  \ldots e^{d}_{-1}f^d$ for $d \geq 5$.
  Also, that maximal conformal weight is exactly $2(s_d - 1)$, so it
  follows that
  $$(\overline{\omega})^{s_d} = \overline{0}$$
  in $\mathcal P(SF(d)^+)$.
  It remains to show that $(\overline{\omega})^{s_d - 1}$ isn't equal to zero
  in $\mathcal P(SF(d)^+)$.
  For $d \leq 4$, we can check by the computer (using relations proven
  in the subsection \ref{subsec:lem}) that we have:
  \begin{align*}
    (\overline{\omega})^4 &= (\overline{h^{11}})^4  \text{ in } \mathcal P(SF(1)^+) \\
    (\overline{\omega})^4 &= \frac{16}{5}(\overline{h^{11}})^4  \text{ in } \mathcal P(SF(2)^+) \\
    (\overline{\omega})^4 &= \frac{37}{5}(\overline{h^{11}})^4  \text{ in } \mathcal P(SF(3)^+) \\
    (\overline{\omega})^4 &= \frac{72}{5}(\overline{h^{11}})^4 + 24\overline{e^1_{-1}f^1_{-1} \ldots e^4_{-1}f^4_{-1}} \text{ in } \mathcal P(SF(4)^+).
  \end{align*}
  By relation \eqref{efef} we get that
  $$\overline{h^{11}}^4 = 360 \overline{e^1_{-7}f^1}.$$
  Now we see that in these cases we can write $(\overline{\omega})^4$
  as a nonzero linear combination of basis elements (by the part 1
  of the Corollary \ref{fcor}).
  
  For $d \geq 5$, one can show that
  $$(\overline{\omega})^d = d! \overline{e^1_{-1}f^1_{-1} \ldots e^d_{-1}f^d} \text{ in } \mathcal P(SF(d)^+).$$
  The strategy is to write:
  $$(\overline{\omega})^d = \sum_{k_1 + k_2 + \ldots + k_m} \frac{d!}{k_1! k_2! \ldots k_m!} \prod_{i = 1}^m (\overline{h^{ii}})^{k_i},$$
  and then (using relations from Lemma \ref{lemlem}) show
  that we have
  $$\prod_{i = 1}^m (\overline{h^{ii}})^{k_i} = \overline{0}$$
  for all choices of $k_i$ except the $k_1 = k_2 = \ldots = k_d = 1.$
\end{proof}

We see that for all $d > 1$ we have $s_d < \deg m_d$.
We can explain this for $d = 2$. Take a $\mathfrak{sp}(2d)$-invariant
element of conformal weight $4$:
$$J^4 = (e^1_{-3}f^1 - f^1_{-3}e^1) + (e^2_{-3}f^2 - f^2_{-3}e^2). $$
One can show that in $A(SF(2)^+)$ we have
\begin{equation} \label{j4}
[J^4] = -\frac{144}{5} [\omega]^5 + 24 [\omega]^4 + \frac{29}{5}[\omega]^3.  \end{equation}
It follows that we can write $[\omega]^5$ as a linear combination
of elements of lower degree, and by part 2 of the Corollary \ref{fcor}
we must have $(\overline{\omega})^5 = \overline{0}$.

In \cite[Section 3]{CL}, the authors showed that the VOA $SF(d)^{\mathfrak{sp}(2d)}$ is a $\mathcal W$-algebra of type $\mathcal W(2,4,\ldots, 2d)$, generated by a primary field of conformal weight $4$.   Later, it was proved in \cite{KL} that  $SF(d)^{\mathfrak{sp}(2d)}$  is isomorphic to the simple principal $\mathcal W$--algebra $\mathcal W_k(sp(2d), f_{prin})$ for $k = - d-1/2$.
Specially, $SF(2)^{\mathfrak{sp}(4)}$
is strongly generated by $\omega$ and $J^4$, and it follows that
$A(SF(2)^{\mathfrak{sp}(4)})$ is generated by $[\omega]$ and $[J^4]$.
We have the induced algebra homomorphism
$$A(SF(2)^{\mathfrak{sp}(4)}) \to A(SF(2)^+),$$
and  it is easy to see that the image of this homomorphism is the invariant subalgebra
$A(SF(2)^+)^{\mathfrak{sp}(4)}$. 
Equation \eqref{j4} says that $A(SF(2)^+)^{\mathfrak{sp}(4)}$ is exactly the subalgebra of $A(SF(2)^+)$ generated by $[\omega]$.

We can show that this also happens to $\mathfrak{sp}(2d)$-invariant
elements in $A(SF(d)^+)$ for all $d$.

\begin{prop} \label{propinv}
  Invariant subalgebra $A(SF(d)^+)^{\mathfrak{sp}(2d)}$ is generated by
  $[\omega]$, and thus isomorphic to $\mathbb C[x]/(m_d(x))$.
\end{prop}
\begin{proof}
  Let $W$ be a subalgebra of $A(SF(d)^+)$ generated by $[\omega]$.
  Because $\omega$ is invariant under the action of $\mathfrak{sp}(2d)$,
  we have $W \leq A(SF(d)^+)^{\mathfrak{sp}(2d)}$.
  Also, because of the part 3 of Corollary \ref{fcor}, we have
  $$W \simeq \mathbb C[x]/(m_d(x)), \ \dim_{\mathbb C} W = d+4.$$
  We will prove that $\dim_{\mathbb C} A(SF(d)^+)^{\mathfrak{sp}(2d)} = d + 4$,
  and the rest will follow.
  We can write
  $$A(SF(d)^+) = A_{-\frac{d}{8}} \oplus A_0 \oplus A_{-\frac{d}{8} + \frac{1}{2}} \oplus A_{1}, $$
  where
  $$A_{\lambda} = \{ x \in A(SF(d)^+) : ([\omega] - \lambda[\mathbf 1])^N x = 0, \text{ for N large enough } \}.$$
  Theorem \ref{uvod1} says that we have
  $$A_{-\frac{d}{8}} \simeq \mathbb C, \ A_0 \simeq \Lambda^{\text{even}}(V_{2d}), \ A_{-\frac{d}{8} + \frac{1}{2}} \simeq A_1 \simeq M_{2d}(\mathbb C). $$
  We want to show that $A_\lambda$
  are in fact $\mathfrak{sp}(2d)$-submodules.
  By the general theory of associative algebras (for the reference,
  see \cite{Lam}), there
  are central idempotent elements $e_\lambda \in A(SF(d)^+)$
  such that
  $$A_\lambda = A(SF(d)^+) e_\lambda. $$
  It is easy to see that these central idempotent elements are
  polynomials in $[\omega]$ by looking at the decomposition
  $$\mathbb C[x]/(m_d(x)) \simeq \mathbb C[x]/(x^{d+1}) \oplus \mathbb C[x]/(x-1) \oplus \mathbb C[x]/(x + d/8) \oplus \mathbb C[x]/(x + d/8 - 1/2). $$
  It follows that each $A_\lambda$ is a $\mathfrak{sp}(2d)$-submodule.

  It is obvious that $A_{- \frac{d}{8}}$ is a trivial $\mathfrak{sp}(2d)$-module.
  By looking at top components of $SF(d)^-$ and $SF(d)_\theta^-$ one can see
  that both $A_1$ and $A_{-d/8 + 1/2}$ have a linear basis of the form
  $$e_{\lambda} * [x], \ \lambda = 1, -d/8 + 1/2$$
  where $x$ is a strong generator of $SF(d)^+$.
  It is easy to check that the linear span of large generators in $SF(d)^+$ is
  a $\mathfrak{sp}(2d)$-submodule isomorphic to $\Sym^2(V_{2d})$, and that the linear span
  of small generators is a $\mathfrak{sp}(2d)$-submodule isomorphic to $\Lambda^2(V_{2d})$.
  Now, by \cite{FH} we have that $\Sym^2(V_{2d})$ is an irreducible $\mathfrak{sp}(2d)$-module,
  and that $\Lambda^2(V_{2d})$ decomposes as
  $$\Lambda^2(V_{2d}) \simeq U \oplus \mathbb C,$$
  where $U$ is an irreducible $\mathfrak{sp}(2d)$-module (the trivial part corresponds to
  $\mathbb C \omega$).
  Now, we have that
  $$ \dim_{\mathbb C} A_1^{\mathfrak{sp}(2d)} =  \dim_{\mathbb C} A_{-d/8+1/2}^{\mathfrak{sp}(2d)} = 1. $$
  For $A_0$, it is easy to see that it has a linear basis of the form
  $$e_0 * [x], x \in B_d^1,$$
  and that we have
  $$A_0 \simeq \bigoplus_{k = 0}^d \Lambda^{2k}(V_{2d}). $$
   Next we use a Theorem  in \cite[Chapter 11.6.7]{Pro} which says that
  for $0 \leq 2k \leq d$ we have
  $$\Lambda^{2k}(V_{2d}) \simeq \bigoplus_{i = 0}^k V_{2d}^{(i)}.$$
  (where by $V_{2d}^{(0)}$ we denote the trivial representation, and
  by $V_{2d}^{(i)}$ the $i$-th fundamental representation).
  It is a well-known fact that for $d < 2k \leq 2d$ we have
  $$\Lambda^{2k}V_d \simeq \Lambda^{2d - 2k} V_d.$$
  To conclude, for $0 \leq 2k \leq 2d$ we have
  $$\dim_{\mathbb C} (\Lambda^{2k} V_d^{\mathfrak{sp}(2d)}) = 1,$$
  which implies
  $$\dim_{\mathbb C} (A_0^{\mathfrak{sp}(2d)}) = d + 1.$$
  Summing up, we get
  $$\dim_{\mathbb C} ((A(SF(d)^+)^{\mathfrak{sp}(2d)}) = d + 4. $$
\end{proof}

\begin{rem}
Note that for $d >1$, the invariant subalgebra  $(A(SF(d)^+)^{\mathfrak{sp}(2d)}  $  is a proper subalgebra of the center  of $A(SF(d)^+)$. In the case of the  triplet vertex algebra $\mathcal W(p)$ (recall that $SF(1) ^ + = \mathcal W(2)$), the center of  $A(\mathcal W(p))$ coincides with the invariant subalgebra  $(A(\mathcal W(p)) )^{\mathfrak{sl}(2)}) $ (cf. \cite{AM}, \cite{AM11}). In our opinion, the reason for this difference is in the fact that $SF(d)^+$ for $d >1$ contains  a large number of generators of conformal weight $2$ which can contribute to the center of $(A(SF(d)^+)$, while weight two space of $\mathcal W(p)$ is $1$-dimensional.

In the case of higher rank triplet vertex algebra $\mathcal W(p)_Q$ we expect that center of $A(\mathcal W(p)_Q)$ conicides with invariants $A(\mathcal W(p)_Q)^{\mathfrak{g}}$, where $\mathfrak g$ is the simple Lie algebra having root system $Q$.
\end{rem}

\section*{Acknowledgment}
 
 The results of this paper were    reported  by  A.\v C.  at the conference Representation Theory XVI, Dubrovnik, June 23-29.2019.
 
The authors  are  partially supported   by the
QuantiXLie Centre of Excellence, a project cofinanced
by the Croatian Government and European Union
through the European Regional Development Fund - the
Competitiveness and Cohesion Operational Programme
(KK.01.1.1.01.0004).

\Addresses

\end{document}